\documentclass[11pt]{amsart}
\usepackage{amssymb, amsfonts, amsmath, amsthm}
\usepackage{enumerate}

\theoremstyle{plain}
\newtheorem{teo}{Theorem}
\newtheorem{prop}[teo]{Proposition}

\newtheorem{lem}[teo]{Lemma}

\theoremstyle{definition}
\newtheorem{defi}[teo]{Definition}

\theoremstyle{remark}

\input xy
\usepackage[all]{xy}

\newcommand{\z}  {\mathbf{Z}}
\newcommand{\ci} {\mathbf{C}}

\newcommand{\qu} {\mathbf{Q}}

\newcommand{\letr}[1] {\mathcal{1}}

\newcommand{\gq} {G_{\mathbf{Q}}}
\newcommand{\f } {\mathbf{F}}

\DeclareMathOperator{\gl}{GL}

\DeclareMathOperator{\enn}{End}

\DeclareMathOperator{\autt}{Aut}

\DeclareMathOperator{\fr}{Frob}

\def\p{\mathfrak{p}}
\def\a{\mathfrak{a}}

\begin{document}
\title[Tate modules and torsion fields]
{Integral Tate modules and splitting of primes in torsion fields of elliptic curves}

\author{{Tommaso Giorgio} {Centeleghe}}

\address{IWR, Universit\"at Heidelberg, Im Neuenheimer Feld 368, 69120 Heidelberg, Germany}
\email{tommaso.centeleghe@iwr.uni-heidelberg.de}
\urladdr{http://www.iwr.uni-heidelberg.de/groups/arith-geom/centeleghe/}

\maketitle

\begin{abstract}
Let $E$ be an elliptic curve over a finite field $k$, and $\ell$ a prime number
different from the characteristic of $k$. In this paper we consider the problem of finding the
structure of the Tate module $T_\ell(E)$ as an integral Galois representations of $k$. We indicate an explicit
procedure to solve this problem starting from the characteristic polynomial $f_E(x)$ and the $j$-invariant $j_E$ of $E$.
Hilbert Class Polynomials of imaginary quadratic orders play here an important role. We give a global application
to the study of prime-splitting in torsion fields of elliptic curves over number fields.
\end{abstract}

\section{Introduction}	

Let $k$ be a finite field of characteristic $p$ and cardinality $q$, $\bar k$ a fixed algebraic closure of $k$, and $G_k$ the corresponding
absolute Galois group. Let $E$ be an elliptic curve over $k$, $\pi_E:E\to E$ the Frobenius isogeny relative to $k$, and $a_E$ the
``error term'' $q+1-|E(k)|$. For a prime $\ell\neq p$, it is well known that the isogeny invariant of $E$ given by the {\it rational} $\ell$-adic
Tate module $V_\ell(E)=T_\ell(E)\otimes_{\z_\ell}\qu_\ell$ is a semi-simple Galois representation of $k$. Thus its isomorphism class is determined
by the characteristic polynomial
$$f_E(x)=x^2-a_Ex+q$$
of the arithmetic Frobenius $\fr_k\in G_k$, which also uniquely identifies the $k$-isogeny class of $E$ (see \cite{T}, Th\'eor\`eme 1). On the other hand,
the {\it integral} representation given by $T_\ell(E)$ is not an isogeny invariant, and finding its Galois structure
requires, in general, a refinement of the information carried by $f_E(x)$ alone.

In \cite{DT}, Duke and T\'oth study the problem of finding $T_\ell(E)$ from a slightly different perspective. Their main result roughly
says that in order to determine $T_\ell(E)$ in the crucial case where $\enn_k(E)\otimes\qu$ is an imaginary quadratic field, it suffices to know $f_E(x)$ and the order
$\enn_{k}(E)$.
More precisely, out of $f_E(x)$ and the {\it index}
$$b_E=[\enn_k(E):\z[\pi_E]]\in\z_{>0}\cup\{\infty\},$$
they construct an explicit two-by-two matrix with integral entries
that gives the action of ${\rm Frob}_k$ on $T_\ell(E)$ in a suitable $\z_\ell$-basis, for any $\ell\neq p$.

In this paper we go one step further and explain how, in almost all cases, the index $b_E$ can be recovered from the $j$-invariant $j_E$ of $E$ and from $f_E(x)$ using
a procedure involving a family of polynomials $\{\mathcal{P}_D(x)\}_{D\in\z_{\leq 0}}$ associated to singular moduli of elliptic curves (see \eqref{DefPD}).
Furthermore, we offer an alternative proof of the result of Duke and T\'oth which avoids Deuring's Lifting Lemma. Instead, it relies on the observation that, when $b_E$ is finite,
$T_\ell(E)$ is free of rank one over $\enn_k(E)\otimes\z_\ell$ (see \S~\ref{section:index}). This freeness is a special instance of a more general fact on abelian varieties with complex multiplication
by a Gorenstein ring (see \cite{ST}, Remark in \S~4). For this reason, the variant we give is suitable for generalizations of \cite{DT} to the higher dimensional setting
of abelian varieties of $\gl_2$-type.

Lastly, we remark that the method we use to find $b_E$ from $f_E(x)$ and $j_E$ has been known for long time when $E$ is ordinary as a consequence of Deuring's Lifting
Lemma. However, the observation that it remains valid in the supersingular case depends on the nature of the ring $\enn_k(E)$ when $E$ does not have all of its endomorphisms
defined over $k$, and may contain some novelty (see \S~\ref{section:index}).

Before stating the main result of this paper we need some definitions and another piece of notation.

\begin{defi} We say that the elliptic curve $E$ over $k$ is {\it special} if $p \equiv 3\pmod 4$, the degree $[k:\f_p]$ is odd, $a_E=0$, and
$j_E=1728\in k$.
\end{defi}
Let $D\in\z$ be a negative discriminant, i.e., $D<0$ and $D\equiv 0\text{ or }1\pmod 4$, and denote by $\mathcal{O}_D$ the imaginary quadratic order of discriminant $D$, viewed
inside the field $\ci$ of complex numbers.
Let
\begin{equation}\label{DefPD}
\mathcal{P}_D(x)=\prod_{\mathcal{O}_D\subset\enn(\ci/\a)}(x- j_{\ci/\a})
\end{equation}
be the separable, monic polynomial in $\ci[x]$ whose roots are all the $j$-invariants of complex elliptic curves $\ci/\a$ whose endomorphism
rings contain the order $\mathcal{O}_D$. For any given $D$ these values of $j$ are algebraic integers which are permuted by the absolute Galois
group of $\qu$, and so $\mathcal{P}_D(x)\in\z[x]$ (see \cite{Co}, \S~11).
Moreover two such $j$-invariants $j_{\ci/\a}$ and $j_{\ci/\a'}$ lie in the same $\gq$-orbit if and only if $\enn_\ci(\ci/\a)=\enn_\ci(\ci/\a')$.
The irreducible factors of $\mathcal{P}_D(x)$ are usually called Hilbert Class Polynomials of the corresponding imaginary quadratic orders. 

Extend the definition of
$\mathcal{P}_D(x)$ to all $D\leq 0$ by setting $\mathcal{P}_0(x)=0$ and $\mathcal{P}_{D}(x)=1$ for all negative $D$ with $D\equiv 2$ or $3\pmod 4$.
Denote by $\Delta_E$ the discriminant $a_E^2-4q$ of $f_E(x)$, and by $\bar g(x)\in\f_p[x]$ the reduction modulo $p$ of any given polynomial
$g(x)\in\z[x]$. Our main result says:

\begin{teo}\label{MT} Let $E$ be an elliptic curve over $k$. Define
$$\beta_E=\sup_{h>0}\{h:h^2|\Delta_E\hphantom{x}{\rm and}\hphantom{x}\bar{\mathcal{P}}_{\Delta_E/h^2}(j_E)=0\},$$
and
$$\sigma_E=\left\{\begin{array}{cl}
\left(\begin{matrix}
\dfrac{a_E\beta_E-\Delta_E}{2\beta_E}&\dfrac{\Delta_E(\beta_E^2-\Delta_E)}{4\beta_E^3}\\
\beta_E&\dfrac{a_E\beta_E+\Delta_E}{2\beta_E}\\
\end{matrix}\right)&\text{ if }\Delta_E<0,\\
&\\
\left(\begin{matrix}
a_E/2&0\\
0&a_E/2\\
\end{matrix}\right)&\text{ if }\Delta_E=0.\\
\end{array}
\right.$$
The matrix $\sigma_E$ has integral entries and, for any prime $\ell\neq p$, describes the action of $\fr_k$ on $T_\ell(E)$, with respect to a suitable $\z_\ell$-basis,
provided that $\ell$ be odd if $E$ is special. Furthermore, $b_E=\beta_E$ if $E$ is not special, and $b_E=\beta_E/2$ or $\beta_E$ otherwise.
\end{teo}

If $E$ is special and $\ell=2$ then $b_E$ is either $p^m$ or $2p^m$, where $[k:\f_p]=2m+1$. For completeness, the two corresponding possibilities for the action of
$\fr_k$ on $T_2(E)$ are given in \S~\ref{proof}. Together with Panagiotis Tsaknias we implemented a Magma package
which computes $\sigma_E$ from $E$ (see \cite{CT}).

An immediate consequence of the theorem is that if $N$ is a positive integer not divisible by $p$ then $\sigma_E\text{ (mod $N$)}$
describes the action of $\fr_k$ on the $N$-th torsion subgroup $E[N](\bar k)$ of $E$. Compared to \cite{DT}, Theorem \ref{MT} has 
the extra feature of indicating a way to construct $\sigma_E$ from the basic invariants $f_E(x)$ and $j_E$.

From Theorem \ref{MT} we deduce the following global application. Let $K$ be a number field, and $\mathcal{E}$ an elliptic curve over $K$ with
$j$-invariant $j_\mathcal{E}$. If $\p$ is a finite prime of $K$ with residue field $k_\p$ at which $\mathcal{E}$ has good reduction $\mathcal{E}_\p$,
denote by $a_\p$ the error term $|k_\p|+1-|\mathcal{E}_\p(k_\p)|$, and by $\Delta_\p$ the discriminant $a_\p^2-4|k_\p|$ of the characteristic polynomial $f_{\mathcal{E}_\p}(x)$
of $\mathcal{E}_\p$.
If $N$ is a positive integer, let $K(\mathcal{E}[N])/K$ be the extension of $K$ obtained by ``joining the coordinates'' of the $N$-torsion points of
$\mathcal{E}$ (see \S~\ref{global} for more details).

\begin{teo}\label{RL} Let $N$ be a positive integer, and $\p$ a finite prime of $K$ not dividing $N$ and at which
$\mathcal{E}$ has good reduction. If $N=2$, assume furthermore that $\mathcal{E}_\p$ is not special. The prime $\p$ splits completely
in $K(\mathcal{E}[N])/K$ if and only if the following conditions hold:

\begin{enumerate}
\item\label{th3.1} $N^2|\Delta_\p$\text{ and }\;$\mathcal{P}_{\Delta_\p/N^2}(j_\mathcal{E})\equiv 0\pmod\p$;
\item\label{th3.2} $a_\p\equiv 2+\frac{\Delta_\p}{N}\pmod {N^*}$;
\end{enumerate}
where $N^*=N$ if $N$ is odd and $N^*=2N$ if $N$ is even.
\end{teo}

Theorem~\ref{RL} gives an idea of how a reciprocity law in a non-abelian context may appear. It emphasizes the role that the family of polynomials
$\{\mathcal{P_D}(x)\}_{D\leq 0}$ play in the study of prime-splitting in torsion field of elliptic curves over number fields.
More generally, Theorem~\ref{MT} can be used to describe the conjugacy class of $\rho_{\mathcal{E}[N]}(\fr_\p)$ in $\autt(\mathcal{E}[N])\simeq\gl_2(\z/N\z)$,
where $\fr_\p$ is a Frobenius element at $\p$.

The paper is structured as follows: in \S~\ref{section:index} we explain how to compute $b_E$ from $f_E(x)$ and $j_E$, in \S~\ref{proof} we give a proof of the main result, and
\S~\ref{global} contains the details of the global application given by Theorem~\ref{RL}. We will retain throughout the notation established so far.
\section{The Index $b_E$}\label{section:index}
In this section we describe a method for finding the index $b_E$ from $f_E(x)$ and $j_E$ using the
family of polynomials $\mathcal{P}_D(x)$. The proof differs in the ordinary and supersingular cases, even though the statement
is uniform (see Proposition~\ref{index}).

We clarify our terminology on complex multiplication of elliptic curves, which follows \cite{El}. If $\bar\kappa$ is an algebraic closure of a field $\kappa$ then
we say that an elliptic curve $A$ over $\kappa$ \emph{has complex multiplication by an imaginary quadratic order $\mathcal{O}$} (abbreviated CM by $\mathcal{O}$)
if there exists an injection
$$\iota:\mathcal{O}\lhook\joinrel\longrightarrow\enn_{\bar\kappa}(A\times_\kappa\bar\kappa)$$
which is maximal, in the sense that it cannot be extended to a strictly larger imaginary quadratic order $\mathcal{O}'\supsetneq\mathcal{O}$.

We observe that an elliptic curve $E$ over a finite field $k$ has CM by $\enn_{k}(E)$ when this ring is an imaginary quadratic order.
The reason is that if $k'/k$ is any field extension, then the natural inclusion
\begin{equation}\label{baseext}
\enn_k(E)\lhook\joinrel\longrightarrow\enn_{k'}(E\times_k k')
\end{equation}
has torsion free cokernel (see \cite{ST}, \S~4). Furthermore, if $E$ is ordinary then $\enn_k(E)$ is an imaginary quadratic order and the map \eqref{baseext} is an
isomorphism for all $k'/k$. Hence any ordinary $E$ has CM \emph{only} by $\enn_k(E)$.

\begin{prop}\label{disc} Let $D$ be a negative discriminant. The $j$-invariant $j_E$ of $E$ is a root of the reduction of $\mathcal{P}_{D}(x)$ modulo $p$
if and only if $E$ has CM by an imaginary quadratic order $\mathcal{O}$ containing $\mathcal{O}_D$.
\end{prop}

The proposition is well-known to experts, therefore we omit its proof. We shall only say that the ``only if'' part follows from the properties of good reduction of CM
elliptic curves in characteristic zero (see \cite{ST}, \S~5), and the ``if'' part is a consequence of Deuring's Lifting Lemma (see \cite{L}, \S~15.5 Theorem~14).

Consider now the quantity
$$\beta_E=\sup_{h>0}\{h:h^2|\Delta_E\hphantom{x}{\rm and}\hphantom{x}\bar{\mathcal{P}}_{\Delta_E/h^2}(j_E)=0\}$$
introduced in Theorem~\ref{MT}. The main proposition of the section is:

\begin{prop}\label{index} If $E$ is not special then $b_E=\beta_E$.
If $E$ is special, then $\beta_E=2p^m$ and $b_E=\beta_E$ or $\beta_E/2$, where $[k:\f_p]=2m+1$.
\end{prop}

\begin{proof} We shall distinguish three cases
\begin{enumerate}
\item\label{or} $E$ is ordinary;
\item\label{unstable} $E$ is supersingular and $\enn_k(E)\otimes\qu$ is an imaginary quadratic field;
\item\label{stable} $E$ is supersingular and $\enn_k(E)\otimes\qu$ is the definite quaternion algebra over $\qu$
of discriminant $p$.
\end{enumerate}
Case \eqref{unstable} is the only case where $E$ acquires new endomorphisms over a suitable (typically quadratic) finite extension $k'$ of $k$,
and we will refer to it as the {\it unstable supersingular} case. Case \eqref{stable} occurs if and only if $b_E$ is infinite or, equivalently, $\Delta_E=0$; in
this situation the proposition is trivial and we will continue assuming $b_E$ finite.

The Frobenius isogeny $\pi_E$ generates a subring of $\enn_k(E)$ of discriminant $\Delta_E$ and gives rise to an inclusion
$$\iota:\mathcal{O}_{\Delta_E}\lhook\joinrel\longrightarrow\enn_{k}(E).$$
Clearly $b_E$ is the largest $h>0$ such that two conditions hold:
\begin{itemize}
\item there exists a quadratic order $\mathcal{O}\supset\mathcal{O}_{\Delta_E}$ with $[\mathcal{O}:\mathcal{O}_{\Delta_E}]=h$;
\item $\iota$ extends to an inclusion $\iota':\mathcal{O}\hookrightarrow\enn_{k}(E)$.
\end{itemize}
Given this, case \eqref{or} follows immediately from Proposition \ref{disc}, since an ordinary
elliptic curve $E$ has complex multiplication only by $\enn_k(E)$.

We are therefore left with the unstable supersingular case. We proceed with a case-by-case computation of both $b_E$ and $\beta_E$.
The polynomials $f_E(x)$ arising in case \eqref{unstable} are characterized by the conditions $\Delta_E<0$ {\it and} $p$ does not split completely in $\qu(\sqrt{\Delta_E})$ (see \cite{W}, \S~4.1
or \cite{T}, Th\'eor\`eme~1).
All possibilities are listed in the following table, where $r=[k:\f_p]$.
\begin{table}[ht]\label{tabella}
\caption{Unstable supersingular Weil polynomials over $k$.}
{\begin{tabular}{@{}ccccc@{}}
$f_E(x)$&$p$&$r$&$\Delta_E$&$b_E$\\ \hline
$x^2+p^{2m+1}$&-&$2m+1$&$-4p^{2m+1}$&$p^m\,{\rm or}\,2p^m$\\
$x^2+p^{2m}$&$\not\equiv 1\,{\rm mod}\,4$&$2m$&$-4p^{2m}$&$p^m$\\
$x^2\pm p^mx+p^{2m}$&$\not\equiv 1\,{\rm mod}\,3$&$2m$&$-3p^{2m}$&$p^m$\\
$x^2\pm p^{m+1}x+p^{2m+1}$&$2\,{\rm or}\, 3$&$2m+1$&$-(4-p)p^{2m+1}$&$p^m$\\
\end{tabular}}
\end{table}

The values of the index $b_E$ appearing in the last column of the table are readily computed from the corresponding $\Delta_E$,
taking into account that $\enn_k(E)$ is maximal locally at $p$ (see \cite{W}, Theorem~4.2). The point is that, except for the case where
$p\equiv 3\pmod 4$, $r=2m+1$ is odd, and $f_E(x)=x^2+p^{2m+1}$, the order $\z[\pi_E]$ is maximal at every prime $\ell\neq p$,
and thus $\enn_k(E)$ is the maximal order, from which the equality $b_E=p^m$ follows.

On the other hand, if $p\equiv 3\pmod 4$ and
$f_E(x)=x^2+p^{2m+1}$, then $b_E$ is either $2p^m$ or $p^m$, according to whether $\enn_k(E)$ is maximal or sits in the maximal order
with index $2$, respectively. Since both cases do arise for suitable $E$ (\cite{W}, Theorem~4.2), $b_E$ is not constant on the $k$-isogeny class defined by $f_E(x)$
and cannot be determined from $f_E(x)$ alone.

We now turn to the study of $\beta_E$. Since $E$ has Complex Multiplication by $\enn_k(E)$, Proposition~\ref{disc} implies that
\[
\beta_E\geq b_E.
\]
Moreover, if $\enn_k(E)$ is the maximal order, then $\beta_E$ cannot be larger than $b_E$ (for $\Delta_E/h^2$ is not a discriminant if $h>b_E$),
and $\beta_E=b_E$ follows. We are therefore left with showing that $\beta_E=b_E$ when $\enn_k(E)$ is not maximal and $E$ is not special. Notice that in this case $\enn_k(E)$ has discriminant
$-4p$, and $p\equiv 3\pmod 4$. By Proposition~\ref{disc}, it is enough to show that: 

\begin{lem}\label{lemmone} Let $p\equiv 3\pmod 4$, assume that $r=2m+1$ is odd, and let $E$ be an elliptic curve over
$k$ with $f_E(x)=x^2+p^{2m+1}$. If $j_E\neq 1728$, then $E$ has CM by $\mathcal{O}_{-p}$ or $\mathcal{O}_{-4p}$ but not by
both these rings. In the first case $b_E=2p^m$, in the second one $b_E=p^m$.
\end{lem}

\begin{proof} Choose a square root $\sqrt{-p}$ of $-p$ inside the imaginary quadratic field $\mathcal{O}_{-p}\otimes\qu$,
and identify $\z[\pi_E]$ with an order of $\mathcal{O}_{-p}\otimes\qu$ via the map sending $\pi_E$ to $\sqrt {-p}^{2m+1}$.
By the maximality of $\enn_k(E)$ at $p$ already mentioned, the inclusion $\z[\pi_E]\subset \enn_k(E)$ extends to
an embedding $\tau:\mathcal{O}_{-4p}\hookrightarrow \enn_k(E)$, sending $\sqrt{-p}^{2m+1}$ to $\pi_E$. If $\tau$ further extends to an
embedding of the maximal order $\mathcal{O}_{-p}$, then $b_E=2p^m$ and $E$ has CM by $\mathcal{O}_{-p}$, otherwise $b_E=p^m$ and
$E$ has CM by $\mathcal{O}_{-4p}$.

To prove the lemma, we need to show that if $j_E\neq 1728$ then it is not possible to find
two inclusions
\[
\iota_1:\mathcal{O}_{-p}\lhook\joinrel\longrightarrow\enn_{\bar k}(E\times_k\bar k)\;\;\text{ and }\;\;
\iota_2:\mathcal{O}_{-4p}\lhook\joinrel\longrightarrow\enn_{\bar k}(E\times_k\bar k)
\]
which are both maximal, in the sense of the definition
of complex multiplication given at the beginning of the section. Since $1728$ is the only supersingular invariant when $p=3$, we may and
will continue the proof assuming $p>3$.

We will argue in two steps: first we show that the existence of a maximal embedding
\[
\iota:\mathcal{O}\lhook\joinrel\longrightarrow\enn_{\bar k}(E\times_k\bar k),
\]
where $\mathcal{O}=\mathcal{O}_{-p}$ or $\mathcal{O}_{-4p}$, ensures the existence of a $k$-form $E_\theta$ of $E$ which is $k$-isogenous to $E$ and
for which $\enn_k(E_\theta)\simeq \mathcal{O}$. Next, using that $j_E\neq 1728$, we show that the ring
of $k$-endomorphisms of any $k$-form of $E$ is isomorphic to $\enn_k(E)$.

Before carrying out this plan, we make a digression on the study $k$-forms of $E$ (see \cite{Si} for more details).
The absolute Galois group $G_k$ acts in a natural way on the left of the group $\autt_{\bar k}(E\times_k\bar k)$.
A $1$-cocycle $\theta$ of this action defines an elliptic
curve $E_\theta$ over $k$ and an isomorphism
\[
\varphi_\theta:E_\theta\times_k\bar k\stackrel{\sim}{\longrightarrow} E\times_k\bar k
\]
such that,
if $\sigma\in G_k$ is the arithmetic Frobenius of $k$, the isogeny
$\pi_{E_\theta}$ corresponds to $\theta(\sigma)\pi_E$ under the identification
\[
\enn_{\bar k}(E_\theta\times_k\bar k)\simeq\enn_{\bar k}(E\times_k\bar k)
\]
induced by $\varphi_\theta$. In particular, extension of scalars identifies $\enn_k(E_\theta)$
with the subring of $\enn_{\bar k}(E\times_k\bar k)$ given by the centralizer of $\theta(\sigma)\pi_E$. This construction
induces a bijection
\[
H^1(G_k,\autt_{\bar k}(E\times_k\bar k))\stackrel{\sim}{\longrightarrow}\{k\text{-forms of }E\}_{/\sim}.
\]

Since $\pi_E^2=-p^{2m+1}$, the curve $E$ acquires all of its geometric endomorphisms over the degree $2$ extension of
$k$ inside $\bar k$, therefore the Galois action of $G_k$ on $\enn_{\bar k}(E\times_k\bar k)$, which is non-trivial, becomes
trivial when restricted to its index $2$ subgroup. The group $\autt_{\bar k}(E\times_k\bar k)$ is cyclic of order $2$, $4$, or $6$,
because $p>3$. It follows that $\sigma\in G_k$ acts on $\autt_{\bar k}(E\times_k\bar k)$ by inversion, and it is easy to see that
evaluation of cocycles at $\sigma$ induces an isomorphism
\[
H^1(G_k,\autt_{\bar k}(E\times_k\bar k))\stackrel{\sim}{\longrightarrow}\autt_{\bar k}(E\times_k\bar k)/\autt_{\bar k}(E\times_k\bar k)^2,
\]
in particular $H^1(G_k,\autt_{\bar k}(E\times_k\bar k))$ has order two.

Let now $\mathcal{O}$ be either $\mathcal{O}_{-p}$ or $\mathcal{O}_{-4p}$ and let $\iota:\mathcal{O}\hookrightarrow\enn_{\bar k}(E\times_k\bar k)$ be a maximal
embedding. The unique ideal $I_p$ of $\enn_{\bar k}(E\times_k\bar k)$ of reduced norm $p$ is principal and generated
by $\iota(\sqrt{-p})$. Since the reduced norm of $\pi_E$ is $p^{2m+1}$, there exists a unit $u\in\autt_{\bar k}(E\times_k\bar k)$
such that
\[
\iota(\sqrt{-p}^{2m+1})=u\pi_E.
\]
If $\theta$ is the $1$-cocycle of $G_k$ valued in $\autt_{\bar k}(E\times_k\bar k)$ satisfying $\theta(\sigma)=u$, the construction described above leads to a $k$-form $E_\theta$ of $E$
and an isomorphism $\varphi_\theta:E_\theta\times_k\bar k\stackrel{\sim}{\longrightarrow} E\times_k\bar k$ such that $\pi_{E_\theta}$ corresponds to $u\pi_E$ and the ring $\enn_k(E_\theta)$
corresponds to $\iota(\mathcal{O})$, the centralizer of $u\pi_E$ in $\enn_{\bar k}(E\times_k\bar k)$. Therefore
\[
\enn_k(E_\theta)\simeq \mathcal{O},
\]
and the first step of our program is complete.

To prove the second step, observe that the assumption $j_E\neq 1728$ is equivalent to require that
$-1\in\autt_{\bar k}(E\times_k\bar k)$ not be a square. Therefore the $1$-cocycle sending $\sigma$ to $-1$ on the
one hand describes the only non-trivial $k$-form of $E$, on
the other hand it defines an elliptic curve over $k$ whose ring of $k$-endomorphisms is isomorphic to $\enn_k(E)$,
since the centralizer of $-\pi_E$ is the same as that of
$\pi_E$. We conclude that if $j_E\neq 1728$ the two non-isomorphic $k$-forms of $E$ have
isomorphic $k$-endomorphism rings. This completes the proof of the lemma.
\end{proof}

To complete the proof of Proposition~\ref{index} we are only left with showing that if $E$ is special then $\beta_E=2p^m$. Equivalently,
we need to show that any special $E$ has CM by $\mathcal{O}_{-p}$. This was observed by Elkies in \cite{El}, where he considered
the elliptic curve over $\f_p$ given by $y^2=x^3-x$.
\end{proof}

We make the final remark that if $E$ is special, then the value of $b_E$ cannot be determined from the sole knowledge of
$f_E(x)$ and $j_E$. One can show that $b_E=2p^m$ if the two-torsion $E[2]$ of $E$ is all defined over $k$,
and $b_E=p^m$ otherwise, where $[k:\f_p]=2m+1$.

\section{Proof of Theorem \ref{MT}}\label{proof}

After explaining in \S~\ref{section:index} how to find the index $b_E$ from $f_E(x)$ and $j_E$, we are now ready
to prove the main theorem of the paper.

We begin by pointing out the basic fact that the action of $\fr_k$ on $T_\ell(E)$ is the {\it same} as that induced by $\pi_E$ via
functoriality of $T_\ell$. Next, we observe that the theorem is trivial if $E$ is supersingular with all of its geometric endomorphisms defined
over $k$. In fact in this case
$\Delta_E=0$ and $\pi_E$ is equal to multiplication by the integer $a_E/2$.

We continue assuming $\Delta_E<0$, i.e., $b_E$ finite, and prove a lemma on the natural action
of $\enn_k(E)\otimes\z_\ell$ on $T_\ell(E)$.

\begin{lem}\label{free} Assume that the index $b_E$ is finite. For any prime $\ell\neq p$, the Tate module
$T_\ell(E)$ of $E$ is free of rank one over $\enn_k(E)\otimes\z_\ell$.
\end{lem}

\begin{proof}
Functoriality of $T_\ell$ induces an isomorphism
\begin{equation}\label{tateiso}r_\ell:\enn_k(E)\otimes\z_\ell\stackrel{\sim}{\longrightarrow}\enn_{\z_\ell[G_k]}(T_\ell(E)).
\end{equation}
This follows, for example, from a celebrated theorem of Tate on abelian varieties over finite fields (see \cite{Ta2}).\footnote{The
injectivity of $r_\ell$ follows from a general fact on abelian varieties (see \cite{M}, \S~19, Theorem~3). Its surjectivity is easy and
can be proved directly without invoking Tate's theorem.}

The ring $\enn_k(E)\otimes\z_\ell$ is a free $\z_\ell$-module of rank two, and admits a $\z_\ell$-basis of the form
$(1, \pi')$, for some $\pi'\in \enn_k(E)\otimes\z_\ell$. Since $r_\ell$ is an isomorphism, $(1, r_\ell(\pi'))$ is a $\z_\ell$-basis of
$\enn_{\z_\ell[G_k]}(T_\ell(A))$. We deduce that for any $s\in \z_\ell$ the element
\[
r_\ell(\pi')-s\cdot 1
\]
is not divisible by $\ell$ in $\enn_{\z_\ell[G_k]}(T_\ell(E))$. Equivalently, the reduction modulo $\ell$ of $r_\ell(\pi')$ is
an endomorphism of $T_\ell(E)/\ell T_\ell(E)$ which is not given by multiplication by a scalar in $\z/\ell\z$.
This is to say that there exists $t\in T_\ell(E)-\ell T_\ell(E)$ such that
\[
r_\ell(\pi')\cdot t\;\not\in\;\z_\ell\cdot t+\ell T_\ell(E).
\]
By Nakayama's Lemma we have that the pair $(t,r_\ell(\pi')\cdot t)$ is a $\z_\ell$-basis
$T_\ell(E)$, since the mod $\ell$ reductions of its components generate $T_\ell(E)/\ell T_\ell(E)$. It follows that
\begin{equation}\label{isotat}\enn_k(E)\otimes\z_\ell\ni a\longmapsto r_\ell(a)\cdot t\in T_\ell(E)
\end{equation}
is an isomorphism of $\enn_k(E)\otimes\z_\ell$-modules, which shows that $T_\ell(E)$ is free of rank one over $\enn_k(E)\otimes\z_\ell$.
\end{proof}

Let $\sqrt{\Delta_E}\in\enn_k(E)$ be the square root of $\Delta_E$ given by $2\pi_E-a_E$. It is elementary to check that
\[
\pi_E'=(\Delta_E+b_E\sqrt{\Delta_E})/2b_E^2
\]
belongs to $\enn_k(E)$ and the pair $\mathcal{B}_\z=(1, \pi_E')$ is a $\z$-basis of $\enn_k(E)$.
Furthermore, multiplication by $\pi_E$ on $\enn_k(E)$ is given in the coordinates induced by $\mathcal{B}_\z$ by the matrix

\[
\sigma_E'=\left(\begin{matrix}
\dfrac{a_Eb_E-\Delta_E}{2b_E}&\dfrac{\Delta_E(b_E^2-\Delta_E)}{4b_E^3}\\
b_E&\dfrac{a_Eb_E+\Delta_E}{2b_E}\\
\end{matrix}\right).
\]

The same matrix a {\it fortiori} describes multiplication by $\pi_E\otimes 1$ on $\enn_k(E)\otimes\z_\ell$, with respect to the $\z_\ell$-basis
deduced from $\mathcal{B}_\z$. Since, by Lemma~\ref{free}, $T_\ell(E)$ is free of rank one over $\enn_k(E)\otimes\z_\ell$, we conclude that
$\sigma_E'$ describes the multiplication action of $\pi_E$ on $T_\ell(E)$ as well, in the $\z_\ell$-coordinates of a suitable basis.
This completes the proof of Theorem~\ref{MT} when $E$ is not special, for $\sigma_E'=\sigma_E$ by Proposition~\ref{index},
and also gives an alternative proof of the main result of \cite{DT}.

If $E$ is special then $b_E=p^m$ or $2p^m$ where $[k:\f_p]=2m+1$ (see Proposition~\ref{index}), and the matrix $\sigma_E'$ is given by
\begin{equation}\label{m1orm2}
m_1=\left(\begin{matrix}
2p^{m+1}&-p^{m+1}(4p+1)\\
p^m&-2p^{m+1}\\
\end{matrix}\right)\text{ or }
m_2=\left(\begin{matrix}
p^{m+1}&-p^{m+1}(p+1)/2\\
2p^m&-p^{m+1}\\\end{matrix}\right),
\end{equation}
respectively. Moreover, $\beta_E=2p^m$ and hence $\sigma_E=m_2$.
The matrix equality
\[
\left(\begin{matrix}
1&p\\
0&2\\
\end{matrix}\right)
m_1{\left(\begin{matrix}
1&p\\
0&2\\
\end{matrix}\right)}^{-1}=m_2
\]
shows that $m_1$ and $m_2$ define the same $\gl_2(\z_\ell)$-conjugacy class for any odd prime $\ell$. This suffices to complete
the proof of Theorem~\ref{MT}.

\section{A Global Application}\label{global}

Let $K$ be a number field, $\bar K$ an algebraic closure of it, and $G_K$ the absolute Galois group ${\rm Gal}(\bar K/K)$
of $K$. If $\p$ is a finite prime of $K$, denote by $k_\p$ its residue field, by $p$ its residual characteristic, by $K_\p$ the corresponding completion of $K$,
and by $G_{K_\p}$ the decomposition group of $G_K$ at $\p$ with respect to the choice of a prime $\bar\p$ of $\bar K$ lying above $\p$.
Denote moreover by $\bar k_\p$ the algebraic closure of $k_\p$ given by the residue field of $\bar\p$, and by $G_{k_\p}$ the Galois group
${\rm Gal}(\bar k_\p/k_\p)$.

Let $\mathcal{E}$ be an elliptic curve over $K$ with $j$-invariant $j_\mathcal{E}$. If $\p$ is a finite prime of $K$ at which $\mathcal{E}$ has good reduction $\mathcal{E}_\p$,
denote by $a_\p$ the error term $|k_\p|+1-|\mathcal{E}_\p(k_\p)|$, and by $\Delta_\p$ the discriminant $a_\p^2-4|k_\p|$. If $N$ is an integer $\ge 1$,
the $N$-th torsion subgroup $\mathcal{E}[N]$ is a finite group scheme over $K$ of rank $N^2$ whose group of $L$-valued points, for any $K$-algebra
$L$, is given by $\mathcal{E}[N](L)={\rm Hom}(\z/N\z, \mathcal{E}(L))$. We will identify $\mathcal{E}[N]$ with $\mathcal{E}[N](\bar K)$, an abelian group isomorphic to ${(\z/N\z)}^2$
equipped with a continuous action of $G_K$. By Galois Theory, the kernel of the representation
$$\rho_{\mathcal{E}[N]}:G_K\longrightarrow\autt(\mathcal{E}[N])\simeq\gl_2(\z/N\z)$$
defines a Galois extension $K(\mathcal{E}[N])/K$, known as the $N$-th torsion field of $\mathcal{E}$, with Galois group isomorphic to
${\rm Im}(\rho_{\mathcal{E}[N]})$. As is well known, $\rho_{\mathcal{E}[N]}$ is unramified
at every finite prime $\p$ of $K$ not dividing $N$ and at which $\mathcal{E}$ has good reduction. More precisely, for any $N$ not divisible by $p$,
the reduction map induces an identification
\begin{equation}\label{iden}\mathcal{E}[N](\bar K)=\mathcal{E}_\p[N](\bar k_\p)
\end{equation}
which is equivariant with respect to the Galois actions of $G_{K_\p}$ and $G_{k_\p}$ (see \cite{ST}, Lemma~2).
Theorem \ref{MT} can then be applied to describe the conjugacy class of
\[
\rho_{\mathcal{E}[N]}({\rm Frob}_\p)
\]
in $\gl_2(\z/N\z)$ in terms
of the trace $a_\p$, the size of $k_\p$, and the $j$-invariant $j_{\mathcal{E}_\p}$, provided that $N$ satisfies the usual constraint
of being odd if $\p$ is one of the finitely many primes of $K$ such that $\mathcal{E}_\p$ is special. Notice that our local method will not
say anything about the ${\rm Im}(\rho_{\mathcal{E}[N]})$-conjugacy class of $\rho_{\mathcal{E}[N]}({\rm Frob}_\p)$. However, if $E$ has no complex multiplication then
a celebrated result of Serre says that ${\rm Im}(\rho_{\mathcal{E}[N]})=\autt(\mathcal{E}[N])$ for all $N$ relatively prime to some positive constant $N_E$ (see \cite{S}).

In Theorem~\ref{RL}, we only made explicit the necessary and sufficient condition for $\rho_{\mathcal{E}[N]}(\fr_\p)$ to act as the identity on
$\mathcal{E}[N]$ or, equivalently, for $\p$ to split completely in the $N$-th torsion field of $\mathcal{E}$.

\begin{proof}[proof of Theorem~\ref{RL}] The theorem is trivial if $\Delta_\p=0$, therefore we continue assuming $\Delta_p<0$ and first treat the case where
$\mathcal{E}_\p$ is not special. The integral matrix associated to $\mathcal{E}_\p$ by Theorem~\ref{MT} is
\[
\sigma_\p=\left(\begin{matrix}
\dfrac{a_\p\beta_\p-\Delta_\p}{2\beta_\p}&\dfrac{\Delta_\p(\beta_\p^2-\Delta_\p)}{4\beta_\p^3}\\
\beta_\p&\dfrac{a_\p\beta_\p+\Delta_\p}{2\beta_\p}\\
\end{matrix}\right),
\]
where $\beta_\p$ is equal to
\[
\sup_{h>0}\{h:h^2\mid\Delta_\p\text{ and }\;\mathcal{P}_{\Delta_\p/h^2}(j_\mathcal{E})\equiv 0\pmod\p\},
\]
which is an integer since $\Delta_\p\neq 0$. Moreover, the ratio $\Delta_p/\beta_p^2$ is necessarily a negative
discriminant, since $\mathcal{P}_D(x)$ is defined as the constant polynomial $1$ for all negative integers $D\equiv 2\text{ or }3\pmod 4$.

By Theorem~\ref{MT}, if $N$ is an integer not divisible by $p$, the prime $\p$ splits completely in $K(\mathcal{E}[N])/K$ if and only if
\begin{equation}\label{congr}
\sigma_\p\equiv\text{Id}_2\pmod N,
\end{equation}
where $\text{Id}_2$ is the two-by-two identity matrix. Our task is verifying that \eqref{congr} is equivalent
to have conditions \eqref{th3.1} and \eqref{th3.2} of Theorem~\ref{RL} both satisfied.

We begin by observing that $\sigma_\p$ is a scalar matrix if and only if condition \eqref{th3.1} holds, which is to say
\begin{equation}\label{iff1}
\sigma_\p\text{ is a scalar matrix }\iff N\mid b_\p.
\end{equation}
The ``only if'' direction of \eqref{iff1} is clear, and to see the ``if'' part it is enough to observe that the upper right entry of $\sigma_\p$ is divisible by
$\beta_\p$, since the ratio $\Delta_\p/\beta_\p^2$ is a negative discriminant, and that the difference of the diagonal entries
\[
\dfrac{a_\p\beta_p-\Delta_\p}{2\beta_p}-\dfrac{a_\p\beta_\p+\Delta_\p}{2\beta_\p}=-\dfrac{\Delta_p}{\beta_p}
\]
is also divisible by $\beta_p$.

To complete the proof in the non-special case it suffices to observe the elementary fact that if $N$ divides $\beta_p$ then
\[
\dfrac{a_\p\beta_\p+\Delta_\p}{2\beta_\p}\equiv 1\pmod N\iff a_\p\equiv 2 +\dfrac{\Delta_\p}{\beta_\p}\pmod {N^*},
\]
where $N^*= N$ if $N$ is odd and $N^*=2N$ otherwise.

Assume now that $\mathcal{E}_\p$ is special, and let $N$ be an integer $>2$ and not divisible by $p$. The Frobenius $\rho_{\mathcal{E}[N]}(\fr_\p)$
is described by one of the two matrices in \eqref{m1orm2}, where $m$ is such that $2m+1=[k_\p:\f_p]$. However, both these matrices are not
congruent to the identity modulo $N$, hence $\p$ does not split in $K(\mathcal{E}[N])/K$. On the other hand, condition \eqref{th3.1} of the
theorem is not satisfied, since $N^2$ does not divide $\Delta_\p=-4p^{2m+1}$, and the proof is complete.
\end{proof}

Finally, notice that in the special case where $N=\ell$ is prime, Theorem \ref{RL} gives a criterion for deciding whether or not $\fr_\p$ acts on
$\mathcal{E}[\ell]$ in a semi-simple fashion in the critical case when $\ell|\Delta_\p$, i.e., when such an action has only one eigenvalue.
This problem has been emphasized in \cite{Sh}.

\section*{Acknowledgments}

It is a pleasure to thank Gebhard B\"ockle for all the time spent together discussing this work and related topics. I thank L\"oic Merel for
several discussions we had, which motivated my interest in the questions considered in this paper. I thank Andrew Snowden and
Jared Weinstein for their comments on an earlier version of the manuscript that helped me improve the exposition.
The first version of this work was completed in June 2013 while I was visiting UCLA. I want to thank Chandrashekhar Khare for his
kind invitation there.


\begin{thebibliography}{00}

\bibitem{CT} T. Centeleghe and P. Tsaknias, Integral Frobenius, {\it Magma package available at}
http://math.uni.lu/$\sim$tsaknias/sfware.html

\bibitem{DT} W. Duke and \'A. T\'oth, The splitting of primes in division fields of elliptic curves, {\it Experiment. Math.}
{\bf 11}(4) (2002) 555--565.

\bibitem{Co} D. Cox, {\it Primes of the form $x^2+ny^2$. Fermat, class field theory, and complex multiplication}. John Wiley \& Sons, Inc., New York (1989).

\bibitem{El} N. Elkies, The existence of infinitely many supersingular primes for every elliptic curve over $\qu$,
{\it Invent. math.} {\bf 89} (1987) 561--567.

\bibitem{L} S. Lang, {\it Elliptic functions}. Second edition. Graduate Texts in Mathematics, {\bf 112}. Springer-Verlag, New-York (1987).

\bibitem{M} D. Mumford, {\it Abelian varieties}. Second edition. Tata Institute of Fundamental Research, Bombay (1974).

\bibitem{S} J.-P. Serre, Propri\'et\'es galoisiennes des point d'ordre fini des courbes elliptiques,
{\it Invent. math.} {\bf 15} (1972) 259--331.

\bibitem{ST} J.-P. Serre and J. Tate, Good reduction of abelian varieties,
{\it Annals of Math.} {\bf 88}(3) (1968) 492--517.

\bibitem{Sh} G. Shimura, A reciprocity law in non--solvable extensions, {\it J. Reine Angew. Math.} {\bf 221} (1966) 209--220.

\bibitem{Si} J. H. Silverman, {\it The arithmetic of elliptic curves}. Graduate Texts in Mathematics, {\bf 106}.
Spinger-Verlag, New York (1986).

\bibitem{Ta2}
J. Tate, Endomorphisms of abelian varieties over finite fields, {\it Invent. math.} {\bf 2} (1966) 134--144.

\bibitem{T} J. Tate, Classes d'isog\'enie des vari\'et\'es ab\'eliennes sur un corps fini,
{\it S\'em. Bourbaki} 21e ann\'ee, 1968/69, no 352.

\bibitem{W} W.C. Waterhouse, Abelian varieties over finite fields, {\it Ann. scient. \'Ec. Norm. Sup.},
$4^e$ s\'erie, t. 2 (1969) 521--560.

\end{thebibliography}
\end{document}